\documentclass[leqno, letterpaper, 11pt]{article}

\usepackage{amssymb}
\usepackage{amsmath}
\usepackage{amsthm}
\usepackage{tikz}
\usepackage{fullpage}
\usepackage{mathrsfs}
\usepackage{subfigure}
\usepackage{verbatim}
\usepackage[hyperfootnotes=false,]{hyperref}

\graphicspath{{Figures/}}

     \newcommand{\Prob}[1]{\mathbb{P}\left(#1\right)}

     \newcommand{\ceil}[1]{\left\lceil#1\right \rceil}

\newcommand{\threshold}{\theta}

     \newcommand{\abs}[1]{\left|#1\right|}

     \renewcommand{\span}{\tt{span}}
 \newcommand{\pgood}{p_{\text{\tt good}}}
 \newcommand{\pbad}{p_{\text{\tt bad}}}

\def\Z{\mathbb{Z}}

\newcommand{\bZ}{{\mathbb Z}}

\newcommand{\cO}{\mathcal O}

\newcommand{\Span}{\text{\tt Span}}
     
\numberwithin{equation}{section}
     
\newtheorem{theorem}{Theorem}[section]

\newtheorem{lemma}[theorem]{Lemma}

\newcommand{\prob}[1]{\mathbb{P}\left(#1\right)}
\newcommand{\probsub}[2]{\mathbb{P}_{#1}\left(#2\right)}

 \newcommand{\note}[1]{{\bf \textcolor{blue}
{[#1\marginpar{\textcolor{red}{***}}]}}}

\begin{document}

\begin{center}\Large
{\bf {Bootstrap percolation on \\products of cycles and complete graphs}}
\end{center}

\begin{center}

{\sc Janko Gravner}\\
{\rm Mathematics Department}\\
{\rm University of California}\\
{\rm Davis, CA 95616, USA}\\
{\rm \tt gravner{@}math.ucdavis.edu}
\end{center}
\begin{center}

{\sc David Sivakoff}\\
{\rm Departments of Statistics and Mathematics}\\
{\rm Ohio State University}\\
{\rm Columbus, OH 43210, USA}\\
{\rm \tt dsivakoff{@}stat.osu.edu}
\end{center} 




\begin{abstract} 
Bootstrap percolation on a graph iteratively enlarges a set of occupied sites by adjoining 
points with at least $\theta$ 
occupied neighbors. The initially occupied
set is random, given by a uniform product measure, and we say that spanning occurs if every point 
eventually becomes occupied. 
The main question concerns the critical probability, that is, the minimal
initial density that makes spanning likely. The graphs we consider are
products of cycles of $m$ points and complete graphs of $n$ points. 
The major part of the paper focuses on the case when two factors are complete graphs and one factor 
is a cycle. We identify the asymptotic behavior of 
the critical probability and show that, when $\theta$ is odd, there are two 
qualitatively distinct phases: the transition from low to high probability of spanning 
as the initial density increases is sharp or gradual, depending on the size of $m$. 
\end{abstract}

\let\thefootnote\relax
\let\thefootnote\relax

\section{Introduction}

Given a graph $G=(V,E)$, {\it bootstrap percolation\/} with {\it threshold\/} $\threshold$
is a discrete-time growth process that, starting from 
an initial configuration $\omega \in\{0,1\}^V$, generates an increasing sequence of configurations 
$\omega=\omega_0,\omega_1,\ldots$. Given $\omega_j$, $j\ge 0$, $\omega_{j+1}$ is given by
$$
\omega_{j+1}(v)=
\begin{cases}
1& \text{if $\omega_j(v)=1$ or $\sum_{w \sim  v}\omega_j(w) \geq
\threshold$}\\
0 &\text{else}
\end{cases}
$$
and $\omega_\infty$ is the pointwise limit of $\omega_j$ as $j\to\infty$.
The initial configuration $\omega$ is random; 
$\{\omega(v): v \in V\}$ is a collection of i.i.d.~Bernoulli random 
variables with parameter $p$. The most natural object of study is the event 
$\Span=\{\omega_\infty\equiv 1\}$ that {\it spanning\/} occurs (in which case we also say that 
the initial configuration $\omega$ {\it spans\/} $V$). 
This process was introduced in \cite{CLR}, and has been widely studied since; 
see \cite{AdL, Hol2}
for readable surveys. 

While some of the earliest results are on infinite lattices \cite{vEn, Sch}, 
many of the most interesting questions are formulated for graphs with finite vertex set
$V$, whose size increases to infinity with an integer parameter $n$. Starting from 
the foundational paper \cite{AiL}, an impressive 
body of work addresses the most natural example of $d$-dimensional 
lattice cube $\bZ_n^d$, with $n^d$ vertices and nearest-neighbor edges \cite{Sch, Hol1, HLR, 
GHM, BB, BBM, BBDM}. More recently, some 
work has been done on the Hamming torus $K_n^d$, which, as the Cartesian product of $d$ 
complete graphs of $n$ vertices, has the same vertex set as the lattice cube, but a much larger 
set of edges, which makes many percolation questions fundamentally different \cite{Siv, GHPS, Sli}.  

When $V$ is finite, $\probsub{p}{\Span}$ is a polynomial in $p$ that increases from $0$ to $1$ 
for $p\in [0,1]$. Therefore, for every $\alpha\in [0,1]$ there exists a 
unique $p_\alpha=p_\alpha(n)$ so that 
$\probsub{p_\alpha}{\Span}=\alpha$. Commonly, $p_{1/2}$ is also called 
the {\it critical probability\/}
and is denoted by $p_c$. In our cases, $\Span$ happens a.~a.~s.~for any $p>0$, which 
results in $p_c\to 0$. 
We say that a {\it sharp transition\/} (for the event $\Span$) occurs if $p_\alpha\sim p_c$
for all $\alpha\in (0,1)$,
as $n\to\infty$, or equivalently, 
$$
\probsub{ap_c}{\Span}\to \begin{cases}
0 &a\in(0,1)\\
1 &a\in (1,\infty)
\end{cases}
$$
Sharp transitions results have been proved in remarkable generality \cite{FK}. 
They hold for bootstrap percolation on 
the lattices $\bZ_n^d$, where much more is proved \cite{Hol1, BB, BBM, BBDM}. 

By contrast, {\it gradual transition\/} occurs if there exists a nondecreasing 
continuous function $\phi$ on $(0,\infty)$ with $\phi(0+)=0$ and $\phi(\infty) = 1$, so that
$$
\probsub{ap_c}{\Span}\to \phi(a)
$$
for all $a\in (0,\infty)$. To date, there has been no general investigation 
of this phenomenon, and it is rigorously established on only a few Hamming tori examples: 
$K_n$ with arbitrary $\theta$ (trivially),
$K_n^2$ with arbitrary $\theta$ \cite{GHPS}, $K_n^d$ with arbitrary $d$ and $\theta=2$ \cite{Sli}, 
and $K_n^3$ with $\theta=3$ \cite{GHPS}. (Due to locality of nucleation events, lattice 
examples with gradual transition are somewhat easier to study \cite{GG}.)
Clearly, sharp and gradual transitions are not the only 
possibilities, and indeed we exhibit examples where neither happens; see Theorem~\ref{mixed-thm}. 

\begin{figure}[ht]\label{graph-fig}
\begin{center}
\includegraphics[width=.6\textwidth]{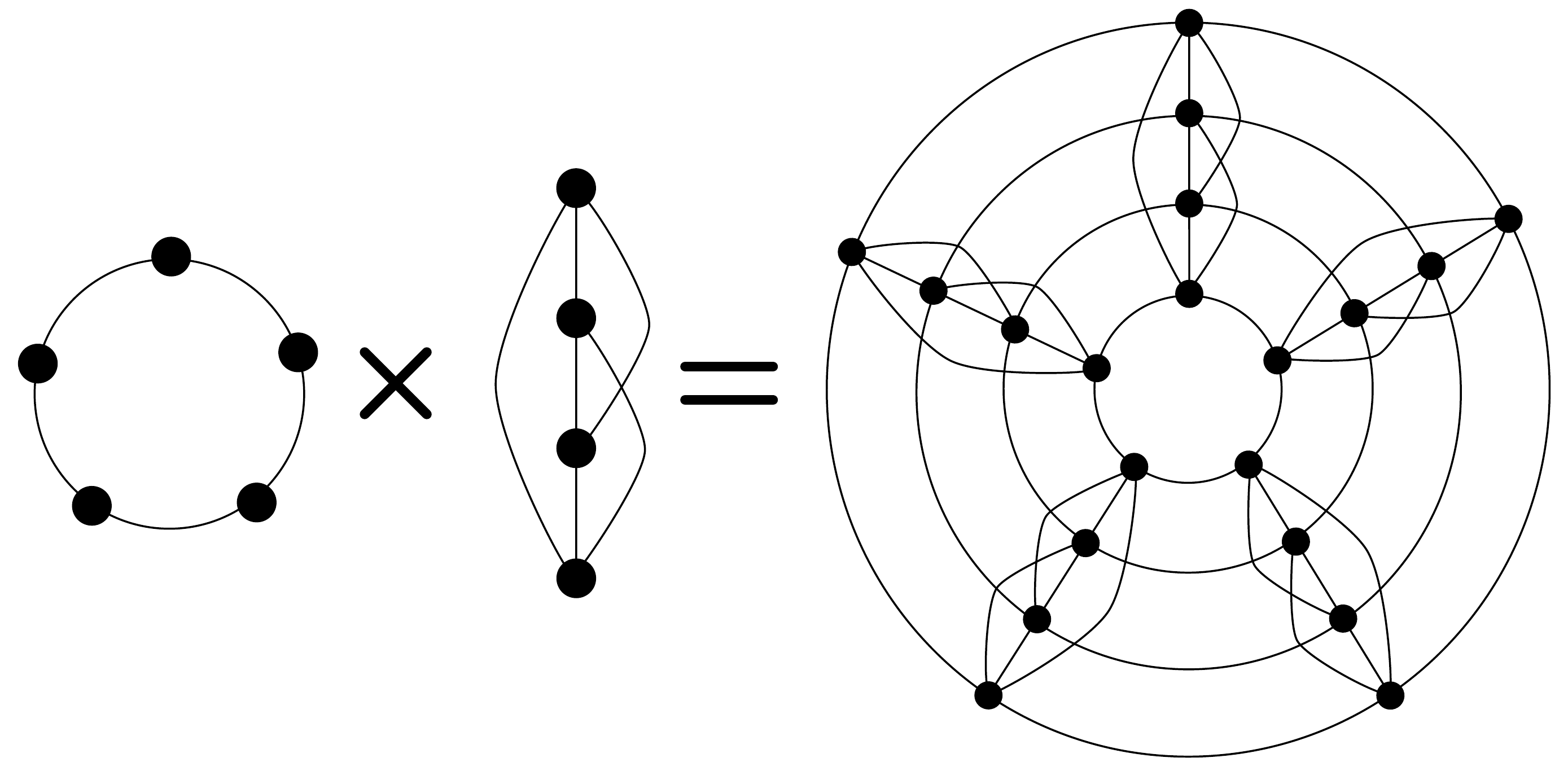}
\end{center}
\caption{Cartesian product of $\bZ_5$ and $K_4$.}
\end{figure}

In this paper, we initiate a study of 
bootstrap percolation on graphs of the form 
$V = (\bZ_m)^{d_1} \times (K_n)^{d_2}$ where $d_1,d_2\geq 1$ are integer parameters, 
$\bZ_m$ denotes the cycle on $m$ vertices, $K_n$ denotes the 
complete graph on $n$ vertices, and $\times$ denotes the Cartesian product (see Figure~\ref{graph-fig} for an example). We will assume throughout that
\begin{equation}\label{m-scaling}
\log m \sim \gamma \log n
\end{equation}
for some constant $\gamma>0$. 
Bootstrap percolation on these graphs may 
be viewed as an extreme case of anisotropic bootstrap percolation, 
where the neighborhood in some directions ($d_1$ of them) is nearest-neighbor, 
but in other directions ($d_2$ of them) the neighborhood extends as far as possible.
The graphs under study could be viewed as limiting cases when the
Holroyd-Liggett-Romik model \cite{HLR} is combined with 
anisotropic graphs studied recently in \cite{DE}. 

We are able to prove a general result for $V = \bZ_m^{d} \times K_n$, which turns out to be mostly  
an application of bootstrap percolation results on lattices \cite{Hol1, BBM, BBDM}. 
Denote by $\log_{(k)}$ the $k$th iterate of $\log$ and let $\lambda(d,\theta)$ be the 
bootstrap percolation scaling 
constant for the lattice $\bZ_m^d$ defined in \cite{BBDM}. 

\begin{theorem} \label{K-theorem} Assume bootstrap percolation on $V = \bZ_m^{d} \times K_n$
for any $d\ge 1$ and $\theta\ge 2$. Then, for $\theta\le d$,
$$
p_c\sim \frac1n \left(\frac{\lambda(d,\theta)}{\log_{(\theta-1)} m}\right)^{d-\theta+1}.
$$
For $\theta>d$,
$$
p_c\sim \frac{d}{2^{\max(2d+1-\theta,0)}}\cdot \frac{\log m}{n} .
$$
In all cases, the transition is sharp. 
\end{theorem} 

The above theorem, whose proof is relegated to Section~\ref{K-section}, 
demonstrates that there is no transition in behaviors for different 
$\gamma$, therefore this case is not of primary interest and we include it mainly for comparison.
Indeed, we will see that the 
situation is very different when the complete graph $K_n$ is replaced by 
the two dimensional Hamming torus $K_n^2$. In this paper we focus on $\bZ_m \times K_n^2$, as
the simplest case that exhibits sharp, gradual and hybrid phase transitions, depending on the relative scaling for $m$ and $n$. We will address the more demanding case $\bZ_m^{d_1} \times K_n^2$ in a subsequent paper, although some of the aforementioned phenomena appear to be limited to the case $d_1=1$.  Higher dimensional Hamming tori $K_n^{d_2}$, $d_2>2$, are much more complex \cite{GHPS, Sli}. We now state our main results.

\begin{theorem} \label{K2-theorem} Assume bootstrap percolation on $\bZ_m \times K_n^2$ and
$\theta\ge 2$. 
\begin{itemize}
\item If $\theta=2\ell+1$ and $\gamma>1/\ell$, then  
\begin{equation}\label{K2-theorem-eq1}
p_c\sim\left({\textstyle\frac12}(\gamma + 1/\ell) \ell! \right)^{1/\ell}\cdot 
\frac{(\log n)^{1/\ell}}{n^{1+1/\ell}},
\end{equation}
with sharp transition.
\item If $\theta=2\ell+1$ and  $\gamma<1/\ell$, then 
\begin{equation}\label{K2-theorem-eq2}
p_c\sim\left( {\textstyle\frac12}\,(\ell+1)!\,{\log 2}\right)^{1/(\ell+1)}\cdot 
\frac{1}{n^{1+1/(\ell+1)}m^{1/(\ell+1)}},
\end{equation}
with gradual transition.
\item If $\theta=2\ell$, then 
\begin{equation}\label{K2-theorem-eq3}
p_c\sim
\begin{cases}
\displaystyle{
\left({\textstyle\frac14}\,\gamma\,\ell!\right)^{1/\ell}\cdot  
\frac{(\log n)^{1/\ell}}{n^{1+1/\ell}}
}
 &\theta\ge 4\\
\displaystyle{{\textstyle\frac 12} \gamma \cdot  
\frac{\log n}{n^{2}}
}&\theta=2,\\
\end{cases}
\end{equation}
with sharp transition. 
\end{itemize}
\end{theorem} 

Therefore, $m\approx n^{1/\ell}$ marks the boundary between sharp and gradual transition 
in case of odd threshold $\theta$, while there is no such boundary when the threshold is even.  
In fact, the odd threshold case has, when $\gamma > 1/\ell$,
another sharp transition. To describe it, call (random) sets $A\subset V$ 
{\it abundant\/} (resp., {\it scarce\/}) if $|A|/|V|\to 1$ (resp., $|A|/|V|\to 0$), in probability,
as $n\to\infty$. Henceforth, 
we make the customary identification of 
the configuration $\omega_\infty$ with the occupied set $\{v:\omega_\infty(v)=1\}$.  

\begin{theorem}\label{almost-span-thm}
Assume bootstrap percolation on $\bZ_m \times K_n^2$, and suppose that 
$\theta = 2\ell + 1$ and $\gamma > 1/\ell$. Assume that 
\begin{equation}\label{p-intro}
p \sim a \cdot \frac{(\log n)^{1/\ell}}{n^{1+1/\ell}}.
\end{equation}
Then $\omega_\infty$ is scarce when $a^\ell/\ell!<1/\ell$ and abundant when $a^\ell/\ell!>1/\ell$.
\end{theorem}
In other words, when the scaling constant $a$ of (\ref{p-intro})
is smaller than $(\ell!/\ell)^{1/\ell}$, then a very small proportion of sites 
becomes open during the bootstrap process.  When $a$ is larger 
than $(\ell!/\ell)^{1/\ell}$ but smaller than $(\frac12(\gamma+1/\ell) \ell!)^{1/\ell}$, 
then most, but not all, of the sites become open.  
Finally, when $a>(\frac12(\gamma+1/\ell) \ell!)^{1/\ell}$, then all sites become open.

Finally, we give the promised example with mixed phase transition, for which we need to 
assume that $m$ satisfies a particular boundary scaling with $\gamma=1/\ell$.  

\begin{theorem}\label{mixed-thm} 
Assume bootstrap percolation on $\bZ_m \times K_n^2$, and suppose that 
$\theta = 2\ell + 1$, that
\begin{equation}
\label{boudary-m-intro}
m\sim\frac{n^{1/\ell}}{(\log n)^{1+1/\ell}},
\end{equation}
and that $p$ satisfies (\ref{p-intro}).  
Then 
\begin{equation}\label{mixed-thm-eq} 
\probsub{p}{\Span}\to
\begin{cases}
0 & a^\ell/\ell!<1/\ell \\
 1-\exp(-2a^{\ell+1}/(\ell+1)!) & a^\ell/\ell!>1/\ell
\end{cases}
\end{equation}

\end{theorem}

The proofs of Theorems~\ref{K2-theorem}--\ref{mixed-thm} are completed in
Section~\ref{sec-exceptional}, after the auxiliary results are established in 
Sections~\ref{sec-preliminary}--\ref{sec-exceptional}. In particular, we use
a result on the birthday problem in a ``slightly supercritical'' regime, which 
is proved in Section~\ref{sec-preliminary}.


\section{Preliminary results}\label{sec-preliminary}

\subsection{Two simple lemmas}

\begin{lemma}\label{lemma-doubleone}
In a sequence of $k$ independent Bernoulli random variables, 
which are $1$ with a small probability 
$r$, 
$$
\Prob{\text{no two consecutive $1$s}}=\exp(-kr^2+\mathcal O(kr^3+r^2)).
$$
\end{lemma}

\begin{proof}
Let $S=\sqrt{1+2r-3r^2}$. 
Then 
\begin{equation*}
\begin{aligned}
\Prob{\text{no two consecutive $1$s}} &= \frac{1+r+S}{2S}\left(\frac {1-r+S}2\right)^k-
& \frac{1+r-S}{2S}\left(\frac {1-r-S}2\right)^k,
\end{aligned}
\end{equation*}
and the result follows by Taylor expansion. 
\end{proof}

\begin{lemma} 
\label{binomial-ld}
For any $p,\epsilon\in (0,1)$ and integer $n$, 
$$
\begin{aligned}
&\prob{\text{\rm Binomal}(n,p)\le  (1-\epsilon)np}\le \exp(-np\epsilon^2/2)\,,\\
&\prob{\text{\rm Binomal}(n,p)\ge  (1+\epsilon)np}\le \exp(-np\epsilon^2/3)\,.
\end{aligned}
$$
\end{lemma}

\begin{proof} 
By exponential Chebyshev, the first probability is for any $\lambda>0$  bounded above 
by 
$$
\exp\left[ \lambda np(1-\epsilon)+n\log(pe^{-\lambda}+1-p)\right]
$$
Taking $\lambda=-\log(1-\epsilon)$, and using $\log(1-\epsilon p)\le -\epsilon p$
and $(1-\epsilon)\log (1-\epsilon)+\epsilon\ge \epsilon^2/2$ gives the desired inequality.
The second probability is for any $\lambda>0$ bounded above by 
$$
\exp\left[ -\lambda np(1+\epsilon)+n\log(pe^{\lambda}+1-p)\right]
$$ 
and we take $\lambda=\log(1+\epsilon)$, and use $\log(1+\epsilon p)\le \epsilon p$
and $(1+\epsilon)\log (1+\epsilon)-\epsilon\ge \epsilon^2/3$.
\end{proof}

\subsection{Birthday Problem}
 
In this self-contained section, we use 
$m$ and $n$ as is customary in the classic birthday problem, therefore 
these variables not have the same meaning as in the rest of the paper.  

The $k$-coincidence birthday problem asks, ``What is the probability that, among $m$ people with birthdays chosen independently and uniformly at random from $[n]=\{1, \ldots, n\}$, there exists a set of $k$ people that have the same birthday?''  Let $A = A_{n,m,k}$ be the event that such a $k$-coincidence exists. 

\begin{lemma}\label{bday-thm} Assume $k$ is fixed, $n$ is large, and $m$ depends on $n$ in 
such a way that $m^{k+1}\ll n^k$. Then 
\begin{equation}\label{bday-asymptotics}
\prob{A^c} \sim \exp\left(-\frac 1{k!}\frac{m^{k}}{n^{k-1}}\right),
\end{equation}
as $n\to\infty$.
\end{lemma}

It is not difficult to show, using Poisson approximation,
that (\ref{bday-asymptotics}) holds when ${m^{k}}/{n^{k-1}}$ approaches a constant
\cite{AGG}. However, we need the formula when $m$ is larger by a multiplicative power of $\log n$,
in which case the standard upper bound for the error in Poisson approximation \cite{BHJ} is too large.
Instead, we use the following asymptotic expansion result. For a function $f$ analytic 
in a neighborhood of $0$, we denote by $\widehat{f}[m]$ the coefficient of $z^m$ in its power 
series expansion.

\begin{theorem}[Theorem 4 in~\cite{Gar}]
\label{Gar_Thm}
Let $f(z) = \sum_{i=0}^{\infty} a_i z^i$, and suppose that $a_0>0$, $a_1>0$, 
$a_i \geq 0$ for $i\geq 2$, and the series 
has positive radius of convergence.  Suppose $m, n\to\infty$ such that $m=o(n)$.  Define $\rho>0$ to be the unique positive solution of 
$$
\rho\frac{f'(\rho)}{f(\rho)}=\frac mn.
$$
Then
$$
\widehat{f^n}[m] = \frac{f(\rho)^n}{\rho^m \sqrt{2\pi m}}(1+o(1)).
$$
\end{theorem} 

\begin{proof}[Proof of Lemma~\ref{bday-thm}]
 We will first give a formula for $\prob{A}$ using embedding into a Poisson process.

Suppose $\{\xi_s^i\}_{i=1}^n$ is a collection of $n$ i.i.d.~Poisson Processes with rate $1/n$, so $\xi_s^i$ is the number of people with birthday $i$ at time $s$, where people arrive at rate $1$.  Therefore, since the distribution of $(\xi_m^1, \ldots, \xi_m^n)$ conditional on $\left\{\sum_{i=1}^n \xi_m^i = m\right\}$ is multinomial, we have
\begin{equation}
\begin{aligned}
 \prob{A^c} &= \prob{\max_{i\in[n]} \xi_m^i \leq k-1 \middle| \sum_{i=1}^n \xi_m^i = m} \\
 &= \frac{\prob{\sum_{i=1}^n \xi_m^i = m, \max_{i\in[n]} \xi_m^i \leq k-1}}{\prob{\sum_{i=1}^n \xi_m^i = m}}  \\
 &= \frac{e^{-m}\left(\frac{m}{n}\right)^n \cdot \widehat{e_k(z)^n}[m]}{\prob{\text{Poisson}(m)=m}}\\
\label{polynomial formula} &\sim  \frac{m!}{n^m}\widehat{e_k(z)^n}[m]
\end{aligned}
\end{equation}
as $m,n\to \infty$, where $e_k(z) = \sum_{i=0}^{k-1} z^i /i!$.

In order to apply Theorem~\ref{Gar_Thm}, we need to estimate $\rho$ and $f(\rho)$ when $f=e_k$. 
We observe that
\begin{equation*}
\begin{aligned}
f(\rho)&=e^\rho-\frac 1{k!}\rho^k +\cO(\rho^{k+1})\\
f'(\rho)&=e^\rho-\frac 1{(k-1)!}\rho^{k-1} +\cO(\rho^{k})\\
\end{aligned}
\end{equation*}
From this, after factoring out $e^\rho$ from both $f(\rho)$ and $f'(\rho)$ and a short computation 
$$
\rho-\frac 1{(k-1)!}\rho^k+\cO(\rho^{k+1})=\frac mn
$$
and then 
$$
\rho=\frac mn+\frac 1{(k-1)!}\frac{m^k}{n^k}+\cO\left(\frac{m^{k+1}}{n^{k+1}}\right).
$$
Next we observe 
$$
\log \left(1+\frac 1{(k-1)!}\frac{m^{k-1}}{n^{k-1}}+\cO\left(\frac{m^{k}}{n^{k}}\right)\right)=
\frac1{(k-1)!}\frac{m^{k-1}}{n^{k-1}}+\cO\left(\frac{m^{k}}{n^{k}}\right), 
$$
and therefore
\begin{equation}\label{rhom}
\rho^m=\frac{m^m}{n^m}\exp\left(\frac1{(k-1)!}\frac{m^{k}}{n^{k-1}}+
\cO\left(\frac{m^{k+1}}{n^{k}}\right)\right).
\end{equation}
Now, 
\begin{equation*}
\begin{aligned}
\log f(\rho)&=\rho+\log\left(1-\frac 1{k!} e^{-\rho} \rho^k +\cO(\rho^{k+1})\right)\\
&=\frac mn+
\left(\frac 1{(k-1)!}-\frac 1{k!}\right)\frac{m^{k}}{n^{k}}+\cO\left(\frac{m^{k+1}}{n^{k+1}}\right)
\end{aligned}
\end{equation*}
and therefore
\begin{equation}\label{frhon}
f(\rho)^n=\exp\left( m+\left(\frac 1{(k-1)!}-
\frac 1{k!}\right)\frac{m^{k}}{n^{k-1}}+\cO\left(\frac{m^{k+1}}{n^{k}}\right)  \right).
\end{equation}
By Theorem~\ref{Gar_Thm}, (\ref{rhom}) and (\ref{frhon}), assuming $m^{k+1}\ll n^k$, 
$$
\widehat{e_k(z)^n}[m] = \widehat{f^n}[m] \sim \frac{e^m n^m}{m^m\sqrt{2\pi m}}
\exp\left(-\frac 1{k!}\frac{m^{k}}{n^{k-1}}\right),
$$
and the asymptotic formula for $\prob{A^c}$ follows from~\eqref{polynomial formula} and Stirling's approximation.
\end{proof}

\section{Single copies of Hamming planes} \label{sec-single}

In the sequel, we will refer to a {\em copy of $K_n^2$}, by which we mean a (deterministic) subgraph $\{i\} \times K_n^2$, for some $i\in \bZ_m$. Its four {\it subsquares\/} are obtained 
by division of $\{i\} \times K_n^2$ (if $n$ is odd) or $\{i\} \times K_{n-1}^2$ (if $n$ is even) into four disjoint congruent squares. 
For a fixed $k\ge 1$, we call a copy of $K_n^2$: 
\begin{itemize}
\item {\it $k$-viable} if it contains a horizontal or a vertical 
line with at least $k$ initially open sites;
\item {\it $k$-internally spanned\/} ({\it $k$-IS\/})
if the bootstrap dynamics with $\theta=k$ restricted to it spans it;
\item {\it $k$-internally inert\/} ({\it $k$-II\/})
if the bootstrap dynamics with $\theta=k$;
restricted to it does not change the initial configuration; 
\item {\it $k$-inert\/} if no site on it becomes open at time $1$
(perhaps with help of neighboring planes); and
\item {\it $k$-proper\/} if within each of the four 
subsquares there are at least $\theta$ horizontal and $\theta$  vertical lines, each containing
at least $k$ points. 
\end{itemize}
In this section, we will assume that $p$ is of the form 
\begin{equation}\label{form-p}
p=a\cdot \frac{(\log n)^{1/\ell}}{n^{1+1/\ell}}, 
\end{equation}
for some $a>0$, where $\ell=\lceil (\theta-1)/2\rceil$. We will now briefly explain why 
this is the critical scaling for spanning, by sketching a simplified argument 
that does not establish a critical constant $a$ but illustrates some of our arguments.  
When $\theta-1$ is odd, the 
probability that a copy of $K_n^2$ is not $\ell$-viable and the probability that it is not
$(\theta-1)$-IS are both about $\exp(-2n(np)^\ell/\ell!)$. When $\theta-1$ is even, the second 
probability goes up to about $\exp(-n(np)^\ell/\ell!)$. In either case, if $a$ is large enough,
every copy of $K_n^2$ is $(\theta-1)$-IS, and all we need for spanning is a single $\theta$-IS 
copy of $K_n^2$, which will appear if $\gamma$ is large. On the other hand, take 
any consecutive Hamming planes $\{i,i+1\}\times K_n^2$, and form a new configuration 
on a Hamming plane $K_n'$ in which 
$x\in K_n'$ is occupied if either $(i,x)$ or $(i+1,x)$ is occupied. If such configuration 
does not span $K_n'$ with threshold $\theta-1$, which certainly happens if such configuration 
is not $\ell$-viable, then the original configuration on $\bZ_m\times K_n^2$ never adds an occupied 
point on $\{i,i+1\}\times K_n^2$ even if all other points become occupied. 
Existence of such a ``blocking pair'' is guaranteed if $a$ is small enough. 

The precise sufficient and necessary conditions for spanning, which yield the correct critical constant in~\eqref{form-p}, are given in Lemmas~\ref{sufficient-condition} and~\ref{necessary-condition}.  To estimate the probabilities that these conditions are met, we need to carefully estimate the probabilities that a single Hamming plane has various internal properties.

\begin{lemma}\label{single-plane-viability}
If $p$ is given by (\ref{form-p}) and $\ell\geq 2$, then 
$$
\probsub{p}{\text{$K_n^2$ is not $\ell$-viable}}\sim n^{-2a^\ell/\ell!}.
$$
\end{lemma}

\begin{proof} 
The probability that a fixed line contains at least $\ell$ open sites is
\begin{equation} \label{spv-eq0}
\frac 1{\ell!}(np)^\ell+\cO(n^{\ell-1}p^\ell+(np)^{\ell+1})=
\frac 1{\ell!}(np)^\ell+\cO((np)^{\ell+1}),
\end{equation}
so the probability of the event $H$ that at least one of $n$ horizontal lines
contains at least $\ell$ open sites satisfies
\begin{equation}\label{spv-eq1}
\probsub{p}{H^c}= \exp\left(-\frac 1{\ell!}n^{\ell+1}p^\ell+\cO(n^{\ell+2}p^{\ell+1})\right)
\sim n^{-a^\ell/\ell!},
\end{equation}
as $n^{\ell+2}p^{\ell+1}$ is $n^{-1/\ell}$ times a power of $\log n$. Conditioned
on $H^c$, the configurations on horizontal lines are independent and the 
conditional probability that any one fixed horizontal line contains exactly one 
open site is bounded below by
$$
q_n=np(1-p)^n=np+\cO((np)^2).
$$
By Lemma~\ref{binomial-ld}, with $\epsilon=C\sqrt{\log n}(nq_n)^{-1/2}$ for a large enough $C$, 
$$
\probsub{p}{\text{fewer than $(1-\epsilon) nq_n$ horizontal lines contain a single 
open site}\,|\, H^c}=o(n^{-a^\ell/\ell!}).
$$
To connect to the birthday problem, focus on these horizontal lines with a single open site. 
The location of the 
open site is uniform on each of these lines, a ``birthday.'' We are looking for 
a {\it vertical\/} line with $\ell$ open sites, which will happen if $\ell$ of these
``birthdays'' coincide. 
Thus, by Lemma~\ref{bday-thm}
\begin{equation}\label{spv-eq2}
\begin{aligned}
&\probsub{p}{\text{no vertical line contains at least $\ell$ open sites}\,|\, H^c}\\
&\le 
(1+o(1))\exp\left(-\frac 1{\ell!}(1-\epsilon)^\ell (nq_n)^\ell/ n^{\ell-1}\right)+
o(n^{-a^\ell/\ell!})\\
&=
(1+o(1))\exp\left(-\frac 1{\ell!}(n^2p)^\ell/n^{\ell-1}+
\cO((\epsilon n^2p+n(np)^2)(n^2p)^{\ell-1})/n^{\ell-1}\right)+
o(n^{-a^\ell/\ell!})\\
&=(1+o(1))\exp\left(
-\frac 1{\ell!}a^\ell\log n+
\cO(\epsilon\log n+np\log n)\right)+
o(n^{-a^\ell/\ell!})\\
&=(1+o(1))n^{-a^\ell/\ell!}.
\end{aligned}
\end{equation}
By FKG inequality, we also have 
$$
\probsub{p}{\text{$K_n^2$ is not $\ell$-viable}}\ge P(H^c)^2
$$
which, together with (\ref{spv-eq1}) and (\ref{spv-eq2}) finishes the proof.
\end{proof}

\begin{lemma}\label{propriety}
Assume that $p$ is given by (\ref{form-p}) and $\ell\geq 1$. 
Then 
\begin{equation}\label{propriety-eq1}
\probsub{p}{\text{$K_n^2$ is not $(\ell-1)$-proper}}=\cO(n^{-L})
\end{equation}
for any $L>0$, while 
\begin{equation}\label{propriety-eq2}
\probsub{p}{\text{$K_n^2$ is not $\ell$-proper}}=o(1).
\end{equation}
\end{lemma}

\begin{proof}
Let $G$ be the event that $K_n^2$ is $(\ell-1)$-proper, 
and $G_1$ the event that the top left (say) square contains $\theta$ horizontal 
lines containing at least $\ell-1$ open points.
As $\probsub{p}{G^c}\le 8P(G_1^c)$, we in fact need to get the upper bound 
in (\ref{propriety-eq1}) for $\probsub{p}{G_1^c}$. Arguing as for (\ref{spv-eq0}), 
the expected number of horizontal lines in the top left square containing at 
least $\ell-1$ open points is $\Omega(n(np)^{\ell-1})$. 
Taking $\epsilon=1/2$ in Lemma~\ref{binomial-ld}, we get 
a constant $c>0$ so that 
\begin{equation}\label{propriety-eq3}
\probsub{p}{G_1^c}\le \exp(-cn^{1/\ell}(\log n)^{1-1/\ell}),
\end{equation}
which is clearly enough for (\ref{propriety-eq1}). To prove (\ref{propriety-eq2}), we use 
analogous definitions of $G$ and $G_1$, and then (\ref{propriety-eq3}) is replaced with
\begin{equation}\label{propriety-eq4}
\probsub{p}{G_1^c}\le \exp(-c\log n), 
\end{equation}
which establishes (\ref{propriety-eq2}).
\end{proof}

\begin{lemma}\label{single-plane-odd-IS}
If $p$ is given by (\ref{form-p}) and $\ell\geq 2$, then 
$$
\probsub{p}{\text{$K_n^2$ is not $(2\ell-1)$-IS}}\sim \probsub{p}{\text{$K_n^2$ is  $(2\ell-1)$-II}}\sim n^{-2a^\ell/\ell!}.
$$
\end{lemma}

\begin{proof}  
Again, let $G$ be the event that $K_n^2$ is $(\ell-1)$ proper. It is easy to see that 
$$
\{\text{$K_n^2$ is $\ell$-viable}\}\cap G\subset \{\text{$K_n^2$ is $(2\ell-1)$-IS}\}
\subset \{\text{$K_n^2$ is not $(2\ell-1)$-II}\}\subset \{\text{$K_n^2$ is $\ell$-viable}\},
$$
and then Lemma~\ref{single-plane-viability} and (\ref{propriety-eq1}) finish the proof.
\end{proof}

\begin{lemma}\label{single-plane-even-IS}
If $p$ is given by (\ref{form-p}) and $\ell\geq 2$, then 
$$
\probsub{p}{\text{$K_n^2$ is not $(2\ell)$-IS}}\sim 
\probsub{p}{\text{$K_n^2$ is $(2\ell)$-II}}\sim 2n^{-a^\ell/\ell!}.
$$
\end{lemma}

\begin{proof} 
Let $H$ (resp., $V$) be the event that at least one horizontal (resp., vertical) line contains 
at least $\ell$ open sites. Yet again, let $G$ be the even that $K_n^2$ is 
$(\ell-1)$-proper. Further, let $F$ be the event
that there exists a line with at least $\ell+1$ open sites.  By the same reasoning as for~\eqref{spv-eq0}, the probability that a fixed line contains at least $\ell+1$ open sites is
$\cO((np)^{\ell+1})$, so
$$
\probsub{p}{F} =\cO(n(np)^{\ell+1}) = \cO(n^{-1/\ell} (\log n)^{1+1/\ell}).
$$
Then 
$$
(H^c\cup V^c)\cap F^c \subset \{\text{$K_n^2$ is $(2\ell)$-II}\},
$$
so by the FKG inequality
\begin{equation}\label{spe-eq1}
\begin{aligned}
\probsub{p}{\text{$K_n^2$ is $(2\ell)$-II}}
&\ge 
\probsub{p}{H^c\cup V^c}\,\probsub{p}{F^c}\\
&=(1-o(1))\,\probsub{p}{H^c\cup V^c}\\
&=
(1-o(1))\,(2\probsub{p}{H^c}-\probsub{p}{H^c\cap V^c})
\\
&= (1-o(1))\,2\probsub{p}{H^c},
\end{aligned}
\end{equation}
by Lemma~\ref{single-plane-viability} and equation~\eqref{spv-eq1}, since $H^c \cap V^c = \{K_n^2 \text{ is not $\ell$-viable}\}$.

On the other hand, with $\circ$ denoting disjoint occurrence, 
$$
(H\circ V)\cap G \subset \{\text{$K_n^2$ is $(2\ell)$-IS}\}
$$
and, by FKG inequality and (\ref{propriety-eq1}),
\begin{equation}\label{spe-eq1.5}
\probsub{p}{\text{$K_n^2$ is $(2\ell)$-IS}}\ge \probsub{p}{H\circ V}\probsub{p}{G}=
\probsub{p}{H\circ V}\, (1-\cO(n^{-L}))
\end{equation}
for any constant $L>0$. To
find an appropriate lower bound for $\probsub{p}{H\circ V}$, 
write $H$ as a disjoint union $H=\cup_{k=1}^n H_k$, where $H_k$ 
is the event that, counted from the top, the $k$th horizontal line 
is the first to contain at least $\ell$ initially open sites. Let $V_k$ be the event that 
there exists a vertical line that contains at least $\ell$ open sites outside of the $k$th horizontal line. Then 
\begin{equation}\label{spe-eq2}
\probsub{p}{H\circ V}\ge \sum_{k=1}^n \probsub{p}{H_k} \probsub{p}{V_k\,|\,H_k}
\end{equation}
Given $H_k$, the horizontal lines are independent, the first $k-1$ have configurations 
conditioned on containing at most $\ell-1$ open sites, and those after the $k$th horizontal line have the independent
Bernoulli configuration. 
An identical argument as in equation~\eqref{spv-eq2}, with the same values of $q_n$ and $\epsilon$, shows 
that 
\begin{equation}\label{spe-eq3}
\probsub{p}{V_k\,|\,H_k}\ge 1-(1+o(1))n^{-a^\ell/\ell!},
\end{equation}
where the $o(1)$ is uniform over all $k$. Then, from (\ref{spe-eq2}), 
\begin{equation}\label{spe-eq4}
\probsub{p}{H\circ V}\ge \left(1-(1+o(1))n^{-a^\ell/\ell!}\right) \probsub{p}{H} \ge 1-2(1+o(1))n^{-a^\ell/\ell!},
\end{equation}
which, together with (\ref{spe-eq1.5}), provides the matching bound to (\ref{spe-eq1})
and ends the proof.
\end{proof}

\begin{lemma}\label{single-plane-above-IS}
If $p$ is given by (\ref{form-p}) and $\ell\ge 1$, then 
$$
\probsub{p}{\text{$K_n^2$ is $(2\ell+1)$-IS}}\sim 
\probsub{p}{\text{$K_n^2$ is not $(2\ell+1)$-II}}\sim 
\frac {2a^{\ell+1}}{(\ell+1)!}\cdot \frac{(\log n)^{1+1/\ell}}{n^{1/\ell}}.
$$
\end{lemma}

\begin{proof} Let now $H$ (resp., $V$) be the event that there exists a horizontal (resp., vertical) line that contains at least
$\ell+1$ open sites. 
We have 
$$
\probsub{p}{H}\sim n\cdot \frac 1{(\ell+1)!} n^{\ell+1}p^{\ell+1}=
\frac {a^{\ell+1}}{(\ell+1)!}\cdot \frac{(\log n)^{1+1/\ell}}{n^{1/\ell}}.
$$
Then, adapting the proof of Lemma 3.3 in \cite{GHPS}, the occurrence of the event
$(H\cap V)\setminus (H\circ V)$ implies that there exists an open site with $2\ell$ additional 
open sites in its Hamming neighborhood, so 
$$
\probsub{p}{(H\cap V)\setminus (H\circ V)}\le n^2\cdot p\cdot n^{2\ell}p^{2\ell}
=\cO(p\cdot (\log n)^2)=o(\probsub{p}{H}^2),
$$
for all $\ell\ge 1$. It then follows from FKG and BK inequalities
that 
$$
\probsub{p}{H\cap V}\sim \probsub{p}{H}^2.
$$ 
Therefore, 
$$
\probsub{p}{\text{$K_n^2$ is not $(2\ell+1)$-II}}\le \probsub{p}{H\cup V} \sim 2\probsub{p}{H}.
$$
On the other hand, let $G$ be the event  that $K_n^2$ is $\ell$-proper. By (\ref{propriety-eq2}),
$P(G)\to 1$. Then 
$$
\probsub{p}{\text{$K_n^2$ is $(2\ell+1)$-IS}}\ge \probsub{p}{(H\cup V)\cap G}\ge 
\probsub{p}{H\cup V}\,\probsub{p}{G}\sim 2\probsub{p}{H},
$$
which establishes the desired asymptotics.
\end{proof}

\begin{lemma}\label{single-plane-below-IS}
If $p$ is given by (\ref{form-p}) and $\ell\geq 1$, then 
$$
\probsub{p}{\text{$K_n^2$ is not $(2\ell-2)$-IS}}=\cO(n^{-L}),
$$
for any constant $L>0$.
\end{lemma}

\begin{proof}
It is easy to see that  
$$\{\text{$K_n^2$ is $(\ell-1)$-proper}\}\subset \{\text{$K_n^2$ is $(2\ell-2)$-IS}\},$$
and so (\ref{propriety-eq1}) finishes the proof.
\end{proof}

\begin{lemma}\label{single-plane-interference}  
If $p$ is given by (\ref{form-p}) and $\ell\geq 1$, then the probability that any fixed copy of 
$K_n^2$ contains a site that has an occupied $\bZ$-neighbor, and at least $\theta-1$ 
occupied $K$-neighbors,
is 
\begin{equation*}
\begin{aligned}
&\begin{cases}
\cO(n^{-1}(\log n)^2) &\theta=2\ell\\
\cO(n^{-1-1/\ell}(\log n)^{2+1/\ell}) & \theta=2\ell+1 
\end{cases}\\
&=o(\probsub{p}{\text{$K_n^2$ is $\theta$-IS}})
\end{aligned}
\end{equation*}
\end{lemma}

\begin{proof} The probability in question is $\cO(n^2p(np)^{\theta-1})$, which implies 
the $\cO$ bounds. Then Lemmas~\ref{single-plane-even-IS} and~\ref{single-plane-above-IS}
finish the proof.
\end{proof}

\begin{lemma}\label{single-plane-interference-below}  
If $p$ is given by (\ref{form-p}) and $\ell\geq 1$, then 
then the probability that any fixed copy of 
$K_n^2$ has a site that that has an occupied $\bZ$-neighbor, and at least $\theta-2$ 
occupied $K$-neighbors,
is 
\begin{equation*}
\begin{aligned}
&
\begin{cases}
\cO(n^{-1+1/\ell}(\log n)^{1-1/\ell}) &\theta=2\ell\\
\cO(n^{-1}(\log n)^{2}) & \theta=2\ell+1 
\end{cases}\\
&=o(\probsub{p}{\text{$K_n^2$ is $\theta$-IS}})
\end{aligned}
\end{equation*}
\end{lemma}
 
\begin{proof} The probability in question is $\cO(1/(np))$ times the one in 
previous lemma.

\end{proof}

\section{Spanning: sufficient condition}\label{sec-sufficient}
 
\begin{lemma}\label{sufficient-condition}
Assume the following three conditions are satisfied:
\begin{enumerate}
\item[(1)] Every copy of $K_n^2$ is $(\theta-2)$-IS.
\item[(2)] Between any pair (if any) of copies of $K_n^2$ that are not $(\theta-1)$-IS
there is a  copy of $K_n^2$ that is  $\theta$-IS.
\item[(3)] There is at least one $\theta$-IS copy of $K_n^2$.
\end{enumerate}
Then $\bZ_m\times K_n^2$ is spanned.
\end{lemma}

\begin{proof} Assume that the graph is not spanned. Then, by (3), there 
is a contiguous interval of Hamming planes that are not fully occupied
in the final configuration; by (1), every such interval is of length at least $2$. 
A boundary plane of this 
interval is not $(\theta-1)$-IS. By (2), 
a plane in the interval must be $\theta$-IS, and thus is fully occupied in the final configuration, 
a contradiction.
\end{proof}

\begin{lemma} \label{odd-theta-supercr} Suppose $\theta=2\ell+1$, $\ell\ge 2$, and $p$ is given by (\ref{form-p}).
Assume that $\gamma>1/\ell$  and $a^\ell/\ell!> (\gamma+1/\ell)/2$. 
Then $\probsub{p}{\Span}\to 1.$
\end{lemma}

\begin{proof} Let $G_1$,  (resp., $G_2$, $G_3$) be the event that
condition (1) (resp,  (2), (3)) above is
satisfied. As $\gamma>1/\ell$, $\probsub{p}{G_3}\to 1$ by Lemma~\ref{single-plane-above-IS}. 
Since $2a^\ell / \ell! > \gamma$, Lemma~\ref{single-plane-odd-IS} implies $\probsub{p}{G_1}\to 1$. 
To show that $G_2$ occurs a.~a.~s., let $\pgood$ be the probability 
that $K_n^2$ is $\theta$-IS, and $\pbad$ the probability that 
$K_n^2$ is not $(\theta-1)$-IS.   
Using 
Lemmas~\ref{single-plane-even-IS} and~\ref{single-plane-above-IS}, 
the lower bound $a^\ell/\ell!> (\gamma+1/\ell)/2 > 1/\ell$ implies that 
$\pbad\ll\pgood$ and $ m\pbad^2/\pgood\to 0$. 

Call a Hamming plane {\it exceptional\/} if
it is either $\theta$-IS or not $(\theta-1)$-IS.
Let 
$K$ be the number of exceptional planes. 
Each of these planes is not $(\theta-1)$-IS with probability $\pbad/(\pgood+\pbad)\sim \pbad/\pgood$.
Moreover, by Lemma~\ref{binomial-ld}, 
\begin{equation}
\label{odd-theta-supercr-eq1}
\probsub{p}{{\textstyle\frac 12}\pgood m\le K\le 2\pgood m}\to 1.
\end{equation}
By Lemma~\ref{lemma-doubleone}, for $\frac 12\pgood m\le k\le 2\pgood m$ and $n$ large enough, 
\begin{equation}
\label{odd-theta-supercr-eq2}
\probsub{p}{G_2\,|\,K=k}\ge \exp(-2k\pbad^2/\pgood^2)\ge \exp(- 4 m\pbad^2/\pgood)\to 1. 
\end{equation}
By (\ref{odd-theta-supercr-eq1}) and (\ref{odd-theta-supercr-eq2}), $\probsub{p}{G_2}\to 1$, 
which finishes the proof.
\end{proof}

\begin{lemma} \label{even-theta-supercr}
Suppose $\theta=2\ell$, $\ell\ge 2$, and $p$ is given by (\ref{form-p}).
Assume that $a^\ell/\ell!> \gamma/4$. 
Then $\probsub{p}{\Span}\to 1.$
\end{lemma}

\begin{proof}
Observe that condition (3) now holds a.~a.~s.~as soon as $m\to\infty$, 
due to Lemma~\ref{single-plane-even-IS}. The rest of the proof is similar to
the one for the previous lemma.
\end{proof}

\begin{lemma}\label{almost-all-spanned}
Suppose $\theta = 2\ell + 1$, $\ell\ge 2$, and $p$ is given by (\ref{form-p}).  
Assume $\gamma > 1/\ell$ and $a^\ell/\ell!> 1/\ell$.  Then $\abs{\omega_\infty}/(mn^2) \to 1$ in probability.
\end{lemma}

\begin{proof}
Let $\gamma'  \in (1/\ell, \gamma)$ be such that $a^\ell / \ell! > (\gamma' + 1/\ell)/2$, and let 
$m'= n^{\gamma'}$.  Divide the cycle, $\bZ_m$, into $m/m'$ intervals of 
length $m'$ (leaving out any leftover interval of smaller length), 
and identify each interval with a subgraph of $\bZ_m \times (K_n)^2$ in the obvious way.  
Denote these subgraphs by $R_1, \ldots, R_{m/m'}$.  
As in the proof of Lemma~\ref{odd-theta-supercr}, call a Hamming plane {\it exceptional\/} if
it is either $\theta$-IS or not $(\theta-1)$-IS.

Observe that a subgraph, $R_i$, is internally 
spanned if it satisfies conditions (1), (2) and (3) of 
Lemma~\ref{sufficient-condition}, in addition to
\begin{enumerate}
\item[(4)] The first and last  
exceptional Hamming planes in $R_i$ are $\theta$-IS.
\end{enumerate}
If $G_4$ is the event that condition (4) is satisfied, 
then Lemmas~\ref{single-plane-even-IS} and~\ref{single-plane-above-IS} imply 
that $\probsub{p}{G_4} \to 1$.  It now follows from the same proof as 
Lemma~\ref{odd-theta-supercr} that $\probsub{p}{R_i \text{ is internally spanned}} \to 1$.  
Let $N$ be the random number of subgraphs among 
$R_1, \ldots, R_{m/m'}$ that are internally spanned, 
and observe that $\abs{\omega_\infty} \geq m'n^2 N$.  Fix $\epsilon>0$.  We have
\begin{equation*}
\probsub{p}{\frac{\abs{\omega_\infty}}{mn^2} > 1-\epsilon} 
\geq \probsub{p}{N > (1-\epsilon) m /m'} \to 1,
\end{equation*}
which proves the claim.
\end{proof}

\section{Spanning: necessary condition} \label{sec-necessary}

For $i_1,i_1\in \bZ_m$, $i_1\ne i_2$, we say that $[i_1,i_2]$ is a {\it blocking interval\/} if 
all Hamming planes $\{i\}\times K_n^2$, $i_1<i<i_2$, are $\theta$-inert, all 
vertices on $\{i_1\}\times K_n^2$ have at most $\theta-2$ initially open neighbors in $\{i_1,i_1+1\}\times K_n^2$, 
and all 
vertices on $ \{i_2\} \times K_n^2$ have at most $\theta-2$ initially open neighbors in $\{i_2-1,i_2\}\times K_n^2$. 
 
\begin{lemma}\label{necessary-condition} Assume that $\bZ_m\times K_n^2$ is spanned. Then both of the following conditions hold:
\begin{enumerate}
\item[(1)] There is no blocking interval. 
\item[(2)] There is at least one copy of $K_n^2$ that is not $\theta$-inert.
\end{enumerate}
\end{lemma}
 
\begin{proof} The necessity of (2) is trivial. Necessity of (1) is also easy, because if (1) fails
and $[i_1,i_2]$ is a blocking interval,  
no point in $[i_1,i_2]\times K_n^2$ gets added even if $([i_1,i_2]\times K_n^2)^c$ is completely 
occupied.
\end{proof}

\begin{lemma}\label{odd-theta-subcr} Suppose $\theta=2\ell+1$, $\ell\ge 2$, 
and $p$ is given by (\ref{form-p}). 
Assume that either $\gamma<1/\ell$, or  
($\gamma>1/\ell$ and $a^\ell/\ell!< (\gamma+1/\ell)/2$). 
Then $\probsub{p}{\Span}\to 0.$
\end{lemma}

\begin{proof}
If $\gamma<1/\ell$, then (2) fails by Lemmas~\ref{single-plane-above-IS} 
and~\ref{single-plane-interference}. 

This time call a Hamming plane $K_n^2$ {\it exceptional\/} if
it is either not $\theta$-II or it is $(\theta-1)$-II, and let $\pgood$ be the probability 
that $K_n^2$ is not $\theta$-II, and $\pbad$ the probability that 
$K_n^2$ is $(\theta-1)$-II. 

Embed the random configuration on $\bZ_m\times K_n^2$ into a random configuration on 
$\bZ_+\times K_n^2$. For any $i\ge 0$, let $\xi_i$ be the random configuration on $[0,i]\times K_n^2$.
Let $I_1$ and $I_2$ be the smallest (random) indices
of two consecutive exceptional planes that are $(\theta-1)$-II. Fix $i_1\ge 0$ and $k\ge 1$. 
Then, conditioned
on $\xi_{i_1-1}$, the event $\{I_1=i_1, I_2-I_1\le k\}$ is a decreasing
function of the configuration on  $[i_1,\infty)\times K_n^2$.

The key to the argument that follows is the event $G_{i_1,i_1+k}$, that there is no neighboring-plane interference in $[i_1,i_1+k]\times K_n^2$. To be more precise, this is the 
event that no vertex in $\{i_1\}\times K_n^2$ has an open neighbor in $\{i_1+1\}\times K_n^2$
together with at least 
$\theta-2$ open neighbors in $\{i_1\}\times K_n^2$; and that no vertex in $\{j\}\times K_n^2$, 
$i_1+1\le j\le i_1+k$, has an open neighbor in $\{j-1\}\times K_n^2$
together with at least $\theta-2$ open neighbors in $\{j\}\times K_n^2$.
Note that $G_{i_1,i_1+k}$ is also decreasing, and independent of the configuration on 
$\xi_{i_1-1}$. Observe also that 
$$
G_{i_1,i_1+k}\cap \{I_1=i_1, I_2\le i_1+k\}\subset 
\{\text{there is blocking interval in $[i_1,i_1+k]$}\}.
$$

We will denote by $\xi$ 
a generic realization of $\xi_{i_1-1}$, and let $k= \log n/(\pgood+\pbad)$.
Conditioned on $\xi_{i_1-1}=\xi$
and $I_1=i_1$, $I_2-i_1$ 
is a geometric random variable with success probability $\pgood+\pbad$, and 
therefore
\begin{equation}
\label{ots-eq0.5}
\probsub{p}{I_2-i_1>k\,|\,\xi_{i_1-1}=\xi, I_1=i_1}\le e^{-k(\pgood+\pbad)} = n^{-1}.
\end{equation}
 We proceed by a similar argument as in the proof 
of Lemma~\ref{odd-theta-supercr}.  Given $\gamma>1/\ell$ and $a^\ell/\ell!< (\gamma+1/\ell)/2$, if $T$ is the number of exceptional planes in $[0,m/2] \times K_n^2$, then by Lemma~\ref{binomial-ld},
\begin{equation}
\label{ots-eq1}
\probsub{p}{T\ge {\textstyle\frac 14}(\pgood+\pbad) m}\to 1.
\end{equation}
By Lemmas~\ref{lemma-doubleone}, \ref{single-plane-even-IS} and \ref{single-plane-above-IS}, for large enough $n$,
\begin{equation*}
\probsub{p}{I_1> m/2\, |\, T\ge {\textstyle\frac 14}(\pgood+\pbad) m} 
\le \exp\left(-{\textstyle\frac18}m\pbad^2 / (\pgood+\pbad)\right) \to 0,
\end{equation*}
because $m\pbad^2 / \pgood \to \infty$ in the case $\pgood > \pbad$, and $m \pbad \to \infty$ otherwise.  Therefore,
\begin{equation}
\label{ots-eq2}
\probsub{p}{I_1\le m/2} \to 1.
\end{equation}
Also, by Lemma~\ref{single-plane-interference-below}, provided $\ell>1$,
\begin{equation}
\label{ots-eq3}
\probsub{p}{ G_{i_1, i_1+k}^c}=\cO(kn^{-1}(\log n)^2)=o(1),
\end{equation}
as $k=\cO(\log n/\pgood) = \cO(n^{1/\ell})$.
Now, 
\begin{equation}
\label{ots-eq4}
\begin{aligned}
&\probsub{p}{\text{$\bZ_m\times K_n^2$ is not spanned}}\\
&\ge 
\probsub{p}{\bigcup_{i_1\leq m/2} \{I_1= i_1, I_2-i_1\le k, G_{i_1, i_1+k}\}}\\
&=
\sum_{i_1\le m/2}\sum_\xi \probsub{p}{\xi_{i_1-1}=\xi} \cdot
\probsub{p}{I_1=i_1, I_2-i_1\le k, G_{i_1, i_1+k}\,|\,\xi_{i_1-1}=\xi}
\\
&\hspace{-5pt}\overset{\text{FKG}}{\ge}
\sum_{i_1\le m/2}\sum_\xi \probsub{p}{\xi_{i_1-1}=\xi} 
\cdot \probsub{p}{I_1=i_1, I_2-i_1\le k\,|\,\xi_{i_1-1}=\xi}
\cdot \probsub{p}{G_{i_1, i_1+k}}
\\
&\hspace{-6pt}\overset{(\ref{ots-eq3})}{=}(1-o(1))\sum_{i_1\le m/2}\sum_\xi \probsub{p}{\xi_{i_1-1}=\xi} 
\cdot \probsub{p}{I_1=i_1, I_2-i_1\le k\,|\,\xi_{i_1-1}=\xi}
\\
&\hspace{-2pt}=\hspace{4pt}(1-o(1))\sum_{i_1\le m/2}\sum_\xi \probsub{p}{\xi_{i_1-1}=\xi, I_1=i_1} 
\cdot \probsub{p}{ I_2-i_1\le k\,|\,\xi_{i_1-1}=\xi, I_1=i_1}
\\
&\hspace{-6pt}\overset{(\ref{ots-eq0.5})}{=}(1-o(1))\sum_{i_1\le m/2}\sum_\xi \probsub{p}{\xi_{i_1-1}=\xi, I_1=i_1} 
\\
&\hspace{-6pt}\overset{(\ref{ots-eq2})}{=}(1-o(1))\probsub{p}{I_1\le m/2} =1-o(1).
\end{aligned}
\end{equation}
This completes the proof of the lemma.
\end{proof}

\begin{lemma} Suppose $\theta=2\ell$, $\ell\ge 2$, and $p$ is given by (\ref{form-p}). \label{even-theta-subcr}
Assume that $a^\ell/\ell!< \gamma/4$. 
Then $\probsub{p}{\Span}\to 0.$
\end{lemma}

\begin{proof}
The argument is similar to that of the previous lemma, and is somewhat simpler, so we omit the details.
\end{proof}

\begin{lemma}\label{almost-inert}
Suppose $\theta = 2\ell+1$, $\ell\ge 2$, and $p$ is given by (\ref{form-p}).  
Assume that $\gamma\ge 1/\ell$ and $a^\ell / \ell! < 1/\ell$.  
Then $\abs{\omega_\infty} / mn^2 \to 0$ in probability.
\end{lemma}

\begin{proof}
Choose constants $\alpha, \beta$ such that $a^\ell/\ell! < \alpha < \beta < 1/\ell$, 
and let $m' = n^\beta$.  As in the proof of Lemma~\ref{almost-all-spanned}, 
divide the cycle $\bZ_m$ into $m/m'$ intervals of length $m'$, 
and denote the resulting subgraphs of $\bZ_m \times K_n^2$ 
by $R_1, \ldots, R_{m/m'}$. Call a site in $R_i$ {\em $\bZ$-assisted} if it has an initially open $\bZ$-neighbor in $R_i$ 
and at least $2\ell-1$ initially open $K$-neighbors

Call a subgraph $R_i$ {\em almost inert} if the following conditions are met.
\begin{enumerate}
\item[(1)] All Hamming planes in $R_i$ are $\theta$-II.
\item[(2)] No sites in $R_i$ are $\bZ$-assisted.
\item[(3)] There is a $(\theta-1)$-II plane among the first $n^\alpha$ Hamming planes and among the last $n^\alpha$ Hamming planes in $R_i$.
\end{enumerate}
If $R_i$ is almost inert, then it contains an interval of Hamming planes of length at least 
$m' - 2n^\alpha = m'(1-o(1))$, in which the initial configuration remains unchanged by the bootstrap dynamics, even if every site outside of the interval becomes open.  The probability that $R_i$ satisfies 
condition (1) converges to $1$ by Lemma~\ref{single-plane-above-IS}, since $\beta<1/\ell$.  
The probability that $R_i$ does not satisfy condition (2) is, 
by Lemma~\ref{single-plane-interference-below}, 
$
\cO(n^{\beta-1}(\log n)^2)=o(1).
$
Finally, the probability that $R_i$ satisfies condition (3) tends to $1$ by 
Lemma~\ref{single-plane-even-IS}.  So we have 
$$\probsub{p}{R_i \text{ is almost inert}} \to 1.$$  
Let $N$ be the number of subgraphs among $R_1, \ldots, R_{m/m'}$ that are {\em not} almost inert, 
and observe that 
$$\abs{\omega_\infty} \leq m'n^2 (N+1)+2n^\alpha (m/m')n^2 + \abs{\omega_0},$$  
and that $m'n^2=o(mn^2)$, $n^\alpha (m/m')n^2=o(mn^2)$.
Fix $\epsilon>0$.  Then, for large enough $n$,  
$$
\probsub{p}{\frac{\abs{\omega_\infty}}{mn^2}>\epsilon} \leq 
\probsub{p}{N > \frac\epsilon3\cdot \frac{m}{m'}} + 
\probsub{p}{\abs{\omega_0} > \frac{\epsilon}{2}\cdot mn^2} \to 0,
$$
where the first term goes to zero because each subgraph $R_1, \ldots, R_{m/m'}$ 
is independently almost inert, and the second term goes to zero because $p\to 0$.  
This completes the proof of the lemma.
\end{proof}

\section{Scaling in the gradual regime}\label{sec-gradual}

Lemma~\ref{odd-theta-subcr} shows that the magnitude of $p$ of the form (\ref{form-p}) is 
too small when $\gamma<1/\ell$ and $\theta=2\ell+1$. In this case, we need to scale $p$ so that 
a $\theta$-IS plane has a chance to appear, so we let
\begin{equation}\label{form-p-grad}
p=a\cdot \frac{1}{n^{1+1/(\ell+1)}m^{1/(\ell+1)}}, 
\end{equation}
where $a>0$ is a constant.

\begin{lemma}\label{single-plane-gradual}
If $p$ is given by (\ref{form-p-grad}),  $\ell\geq 1$, and $\gamma < 1/\ell$, then 
$$
\probsub{p}{\text{$K_n^2$ is $(2\ell+1)$-IS}}\sim 
\probsub{p}{\text{$K_n^2$ is not $(2\ell+1)$-II}}\sim 
\frac {2a^{\ell+1}}{(\ell+1)!}\cdot \frac{1}{m}.
$$
\end{lemma}

\begin{proof} We emulate the proof of Lemma~\ref{single-plane-above-IS}, with the same 
notation. Now,
$$
\probsub{p}{H}\sim  
\frac {a^{\ell+1}}{(\ell+1)!}\cdot \frac{1}{m}
$$
and  
$$
\probsub{p}{(H\cap V)\setminus (H\circ V)}\le n^2\cdot p\cdot n^{2\ell}p^{2\ell}
=\cO(m^{-2}/(n^2p))=o(\probsub{p}{H}^2),
$$
as $\gamma<1/\ell\le 1$. 
The rest of the proof follows from that of Lemma~\ref{single-plane-above-IS}; note that 
$p$ of the form (\ref{form-p-grad}) is larger than that of the form  
(\ref{form-p}) and thus (\ref{propriety-eq2}) holds.  
\end{proof}

\begin{lemma}\label{span-grad}
Assume that $\theta=2\ell+1$ and $\gamma<1/\ell$, and that $p$ is of the form (\ref{form-p-grad}). 
Then $$\probsub{p}{\Span}\to 1-\exp(-2a^{\ell+1}/(\ell+1)!).$$ 
\end{lemma}

\begin{proof} Observe that
$$
\left(\bigcap_{i=0}^{m-1} \left\{ \text{$\{i\}\times K_n^2$ is $(2\ell)$-IS} \right\}\right)\mathbin{\scalebox{1.4}{\ensuremath{\cap}}}
\left(\bigcup_{i=0}^{m-1} \left\{ \text{$\{i\}\times K_n^2$ is $(2\ell+1)$-IS} \right\}\right) \subseteq \Span.
$$
As the size of $p$ given by \eqref{form-p-grad} is much larger than that given by 
\eqref{form-p} we can, by monotonicity in $p$, take the scaling constant $a$ in 
Lemma~\ref{single-plane-even-IS} to be arbitrarily large, 
which implies that 
$$
\probsub{p}{\bigcap_{i=0}^{m-1} \left\{ \text{$\{i\}\times K_n^2$ is $(2\ell)$-IS} \right\}}\to 1.
$$
So, by FKG inequality, 
\begin{equation}
\label{sg-eq1}
\begin{aligned}
\probsub{p}{\Span}
&\ge (1-o(1)) 
\probsub{p}{\bigcup_{i=0}^{m-1} \left\{ \text{$\{i\}\times K_n^2$ is $(2\ell+1)$-IS} \right\}}\\
&\to 1-\exp(-2a^{\ell+1}/(\ell+1)!),
\end{aligned}
\end{equation}
by Lemma~\ref{single-plane-gradual} and elementary Poisson convergence.

Call a site {\em $\bZ$-assisted} if one of the following holds: the site has an initially open $\bZ$-neighbor and at least $2\ell$ initially open $K$-neighbors, or the site has two initially open $\bZ$-neighbors  
and at least $2\ell-1$ initially open $K$-neighbors.  The probability that there exists a $\bZ$-assisted site  
 is bounded above by a constant times 
$$
mn^2\left[p(np)^{2\ell} + p^2 (np)^{2\ell-1}\right]=\cO(m^{-1}/(n^2p))=o(1).
$$
If an initially closed site becomes open in the first step of bootstrap percolation, then either that site is in a plane that is not $(2\ell+1)$-II, or that site is $\bZ$-assisted. 
Therefore, 
\begin{equation}
\label{sg-eq2}
\begin{aligned}
&\probsub{p}{\text{there exists a site that becomes open in the first step}}\\&\le  
\probsub{p}{\bigcup_{i=1}^m \left\{ \text{$\{i\}\times K_n^2$ is not $(2\ell+1)$-II} \right\}}+o(1)\\
&\to 1-\exp(-2a^{\ell+1}/(\ell+1)!),\\
\end{aligned}
\end{equation}
again by Lemma~\ref{single-plane-gradual}. The two asymptotic bounds (\ref{sg-eq1}) and 
(\ref{sg-eq2}) establish the desired convergence.
\end{proof}

\section{Exceptional Cases and Proofs of Main Theorems} \label{sec-exceptional}
 
\subsection{Threshold $2$} \label{subsec-theta-2}
  
We now show that when $\theta=2$, the critical $a$ is twice as large as obtained by 
taking $\ell=1$ in Lemmas~\ref{even-theta-supercr} and~\ref{even-theta-subcr}. 
Thus we will still assume (in this and in the next subsection) 
that $p$ has the scaling given by (\ref{form-p}), that is,
\begin{equation}\label{form-p-1}
p=a\frac{\log n}{n^2}.
\end{equation}

\begin{lemma} \label{single-plane-2-IS}
Assume that $p$ is of the form (\ref{form-p-1}). Then 
$$
\probsub{p}{\text{$K_n^2$ is not $2$-IS}}
\sim \probsub{p}{\text{$K_n^2$ is $2$-II}}
\sim \frac{a\log n}{n^a}.
$$
\end{lemma}

\begin{proof} Let $G_k$ be the event that $K_n^2$ contains exactly 
$k$ initially occupied points. Then 
\begin{equation}\label{sp2-eq1}
\probsub{p}{\text{$K_n^2$ is $2$-II}}\ge \probsub{p}{G_1}\sim n^2pe^{-n^2p}=an^{-a}\log n.
\end{equation}
Moreover,
\begin{equation}\label{sp2-eq2}
\begin{aligned}
&\probsub{p}{\text{$K_n^2$ is not $2$-IS}}\\
&= \probsub{p}{G_0} +\probsub{p}{G_1}\\
&\qquad+\sum_{k\ge 2} \probsub{p}{G_k}\probsub{p}{\text{all initially occupied points lie on the same line}\,|\,G_k}\\
&\le (1+o(1)) an^{-a}\log n+ \sum_{k\ge 2} (n^2p)^k\exp{(-(n^2-k)p)}\cdot\frac{1}{2n-1}\cdot \frac 1{n^{k-2}}\\
&= (1+o(1)) an^{-a}\log n+\frac{e^{-n^2p}}{2n-1}(n^2pe^p)^2(1-npe^p)^{-1}\\
&= (1+o(1)) an^{-a}\log n+ \cO(n^{-a-1}(\log n)^2)
\end{aligned}
\end{equation}
Together, (\ref{sp2-eq1}) and  (\ref{sp2-eq2}) establish the desired asymptotics.
\end{proof}

\begin{lemma} \label{theta-2}
Suppose $\theta=2$, and $p$ has the form (\ref{form-p-1}). 
Then $\probsub{p}{\Span}$
converges to $0$ when $a<\gamma/2$ and to $1$ when $a>\gamma/2$.
\end{lemma}

\begin{proof}
If two neighboring copies of $K_n^2$ are initially empty then spanning cannot occur. By Lemma~\ref{lemma-doubleone}, the 
probability of this is close to $0$ (resp., $1$) if 
$m(1-p)^{2n^2}
$ goes to $\infty$ (resp., to $0$), which happens when $a>\gamma/2$ (resp., when $a<\gamma/2$). 
Thus, when  $a<\gamma/2$, $\probsub{p}{\Span}\to 0$. On the other 
hand, a sufficient condition for spanning is that there are no initially empty neighboring 
copies of $K_n^2$ {\it and\/} there is at least one $2$-IS copy of $K_n^2$. Therefore, Lemma~\ref{single-plane-2-IS} and the FKG inequality imply that, if $a>\gamma/2$,   
$\probsub{p}{\Span}\to 1$.
\end{proof}

\subsection{Threshold $3$} \label{subsec-theta-3}
  
We now handle the case $\theta=3$, 
beginning with the restatement of Lemma~\ref{single-plane-above-IS}
for this case.

\begin{lemma} \label{single-plane-3-IS}
Assume $p$ has the form (\ref{form-p-1}). Then 
$$
\probsub{p}{\text{$K_n^2$ is $3$-IS}}\sim 
\probsub{p}{\text{$K_n^2$ is not $3$-II}}\sim 
a^2\cdot \frac{(\log n)^{2}}{n}.
$$
\end{lemma}

\begin{lemma} \label{theta-3}
Assume $p$ has the form (\ref{form-p-1}). Then 
the conclusions of Lemmas~\ref{odd-theta-supercr}, \ref{odd-theta-subcr}, and~\ref{span-grad} hold, with $\ell=1$.
\end{lemma}

\begin{proof}
Lemmas~\ref{odd-theta-supercr} and~\ref{span-grad} 
are proved in the same fashion, substituting 
Lemma~\ref{single-plane-2-IS} for Lemma~\ref{single-plane-even-IS}. 
Clearly, Lemma~\ref{odd-theta-subcr} holds when $\gamma<1/\ell$. 

The rest of the proof of Lemma~\ref{odd-theta-subcr} needs to be slightly adapted, as 
now the probability in Lemma~\ref{single-plane-interference-below}
is only $\cO(\probsub{p}{\text{$K_n^2$ is $3$-IS}})$, by Lemma~\ref{single-plane-3-IS}. 
However, this means that 
$\probsub{p}{G_{i_1,i_1+k}}\ge \alpha>0$ for some fixed number $\alpha>0$
(instead of converging to $1$), and as a result the probability of occurrence of 
a blocking interval is at least $(1-o(1))\alpha$. This is still true if we replace 
$m$ by $m'=n^{\gamma'}$, where $\gamma'<\gamma$ and 
$a^\ell/\ell!<(\gamma'+1/\ell)/2$. But this means that we have 
have $m/m'\gg 1$ independent possibilities for a blocking 
interval to occur, which is sufficient.
\end{proof}

\subsection{A Boundary Case} \label{subsec-boundary1}

Here, we provide an example with $\theta=2\ell+1$ whereby the transition is neither sharp nor gradual. 
This occurs at a boundary case  $\gamma=1/\ell$; more precisely, we 
assume that 
\begin{equation}
\label{boundary-m}
m=\frac{n^{1/\ell}}{(\log n)^{1+1/\ell}},
\end{equation}
and that $p$ is given by either (\ref{form-p}) or (\ref{form-p-grad}), which now match. 

\begin{lemma}\label{span-boundary}
Assume that $\theta=2\ell+1 \geq 3$, $m$ is given by (\ref{boundary-m}), and $p$ is given by (\ref{form-p}). 
Then (\ref{mixed-thm-eq}) holds.  
\end{lemma}

\begin{proof}
If $a^\ell/\ell!>1/\ell$, then by Lemma~\ref{single-plane-even-IS} (when $\ell>1$) 
or Lemma~\ref{single-plane-2-IS} (when $\ell=1$), 
a.~a.~s.~every copy of $K_n^2$ is $(2\ell)$-IS, and computations (\ref{sg-eq1}) and (\ref{sg-eq2})
apply. On the other hand, if $a^\ell/\ell!<1/\ell$, then
we apply Lemma~\ref{almost-inert}; this lemma also holds for $\ell=1$ because we can apply Lemma~\ref{single-plane-2-IS} in place of Lemma~\ref{single-plane-even-IS} in its proof.
\end{proof}

\subsection{Proofs of Main Theorems} \label{subsec-proofs}

\begin{proof}[Proof of Theorem~\ref{K2-theorem}]
Lemmas~\ref{odd-theta-supercr},~\ref{odd-theta-subcr}, and~\ref{theta-3}  prove (\ref{K2-theorem-eq1}). 
Lemmas~\ref{span-grad} and~\ref{theta-3} prove (\ref{K2-theorem-eq2}). Finally, 
Lemmas~\ref{even-theta-supercr},~\ref{even-theta-subcr}, and~\ref{theta-2}  
prove (\ref{K2-theorem-eq3}).
\end{proof}

\begin{proof}[Proof of Theorem~\ref{almost-span-thm}]
By monotonicity, we may assume (\ref{form-p}) instead of (\ref{p-intro}).
For $\theta>3$, the theorem is then 
clearly a consequence of Lemmas~\ref{almost-all-spanned}
and~\ref{almost-inert}. As the second lemma holds for $\theta=3$, 
we only need to observe that the proof of Lemma~\ref{almost-all-spanned} 
holds when $\theta=3$ by Lemmas~\ref{theta-3} and~\ref{theta-2}. 
\end{proof}

\begin{proof}[Proof of Theorem~\ref{mixed-thm}] This follows from Lemma~\ref{span-boundary} and 
monotonicity. 
\end{proof}

\section{Proof of Theorem~\ref{K-theorem}}\label{K-section}

\begin{lemma}\label{small-lattice-set} Assume that $\theta\in [d+1,2d+1]$.
Suppose $A\subset \bZ^d$ and $|A|<2^{2d+1-\theta}$. Then there is a point in $A$ with at least 
$\theta$ neighbors in $A^c$. 
\end{lemma}

\begin{proof}
We prove this by induction on $d$. For $d=1$ the claim is obvious. 
Assume now the claim holds for $d-1$. To prove it for $d$, first observe that it is trivially true 
for $\theta=2d+1$, or when $A$ consists of a single point. 
Otherwise, let $\underline i=\min\{i:(\bZ^{d-1}\times \{i\})\cap A\ne \emptyset\}$
and $\overline i=\max\{i:(\bZ^{d-1}\times \{i\})\cap A\ne \emptyset\}$. 
We may, without loss of generality, 
assume that $\underline i<\overline i$ (otherwise we permute the coordinates), and that $A'=(\bZ^{d-1}\times \{\underline i\})\cap A$
has cardinality $|A'|<2^{2d+1-\theta}/2=2^{2(d-1)+1-(\theta-1)}$. As 
$\theta-1\in[d,2d-1]=[(d-1)+1, 2(d-1)+1]$, we may apply the induction hypothesis to find
a point in $A'$ with at least $\theta-1$ neighbors in $(\bZ^{d-1}\times \{\underline i\})\cap A^c$, 
but any point in $A'$ also has a neighbor in  $(\bZ^{d-1}\times \{\underline i-1\})\subset A^c$.
\end{proof} 

\begin{proof}[Proof of Theorem~\ref{K-theorem} when $\theta>d$.]
Assume $\theta\in[d+1, 2d+1]$. 
We call a {\it safe box\/} a set of 
the form $\prod_{i=1}^d[a_i,b_i]\times K_n$, where $b_i=a_i+1$ for 
$2d+1-\theta$ indices $i$ and $b_i=a_i$ for the rest. Every vertex in a safe box has exactly $2d-(2d+1-\theta)=\theta-1$ neighbors outside. Therefore, 
if there exist a completely empty safe box, then spanning is  
impossible. It follows that a.~a.~s.~spanning does not occur when 
$$
(1-p)^{n\cdot 2^{2d+1-\theta}}\gg m^{-d},
$$
and this inequality is satisfied if 
$$
p<(1-\epsilon) \frac{d}{2^{2d+1-\theta}}\cdot \frac{\log m}{n},
$$
for some $\epsilon>0$, implying the lower bound in Theorem~\ref{K-theorem}. Conversely, call a point $z\in \bZ_m^d$ {\it white\/} 
if $\{z\}\times K_n$ contains at most $\theta-1$ initially open points. If
$$
p>(1+\epsilon) \frac{d}{2^{2d+1-\theta}}\cdot \frac{\log m}{n},
$$
then the probability that there exists a white connected set of size $k\ge 1$ is 
bounded above by 
$$
C^k m^d  \left((np)^{\theta-1}e^{-np}\right)^k\le 
C^k m^d(\log m)^{k(\theta-1)}m^{-(1+\epsilon)dk2^{-2d-1+\theta}},
$$
where $C$ denotes two diferent constants dependent on $d$ and $\theta$. Clearly, the 
above expression goes to $0$ 
as $n\to\infty$ provided that $k\ge 2^{2d+1-\theta}$. It follows that 
a.~a.~s.~there is no connected white set of size  at least $2^{2d+1-\theta}$ in $\bZ_m^d$. 
By Lemma~\ref{small-lattice-set}, any connected white set of smaller size has at least one point 
$z_0$ that 
has $\theta$ non-white neighbors and thus the entire line $\{z_0\}\times K_n$ becomes occupied by 
time $2$. Thus the entire $V$ becomes occupied by time $2^{2d+1-\theta}$.
The proof for $\theta>2d+1$ is the same as when $\theta=2d+1$. 
\end{proof}

\begin{proof}[Proof of Theorem~\ref{K-theorem} when $\theta\le d$.]
Write $\lambda=\lambda(d,\theta)$, fix an $\epsilon>0$, and assume $p$ is of the form 
$p=an^{-1}(\log_{(\theta-1)}m)^{-(d-\theta+1)}$. 
Call a point $z\in \bZ_m^d$ {\it grey\/} if the line $\{z\}\times K_n$ contains at least one 
open site, and {\it black\/} if the line $\{z\}\times K_n$ contains at least $\theta$ 
open points. It is a necessary condition for spanning that initially grey points span 
$\bZ_m^d$ under the 
bootstrap percolation process with threshold $\theta$.  By the main result of  
\cite{BBDM}, $\probsub{p}{\text{\Span}}\to 0$ unless
$$
np\ge  1-(1-p)^n\ge 
(1-\epsilon) \left(\frac {\lambda}{\log_{(\theta-1)} m}\right)^{d-\theta+1},
$$
which proves that $\probsub{p}{\text{\Span}}\to 0$ when $a<\lambda^{d-\theta+1}$.
To get the upper bound, observe that 
\begin{equation}\label{K-thm-eq1}
\probsub{p}{\text{$z$ is initially black}}=
\Omega\left(\probsub{p}{\text{$z$ is initially grey}}^\theta\right).
\end{equation}
Next, observe that if a point $z\in \bZ_m^d$ has $\theta-1$ black points in its neighborhood
and an additional grey point (that could be $z$ itself), then $z$ eventually becomes black. 
The local growth of black points can then be constructed in the same way 
as in \cite{BBM}, and due to (\ref{K-thm-eq1}) has the same 
leading order asymptotics as the local growth probability of grey points, as the main contribution 
comes from a lower-dimensional process where, apart from an initial black nucleus, only 
grey points are used. This shows that $\probsub{p}{\text{spanning}}\to 1$ 
when $a>\lambda^{d-\theta+1}$.
\end{proof}

\section*{Acknowledgements}
Janko Gravner was partially supported by the Simons Foundation Award \#281309
and the Republic of Slovenia's Ministry of Science
program P1-285.  David Sivakoff 
was partially supported by NSF CDS\&E-MSS Award \#1418265.

\end{document}